\documentclass[11pt]{amsart}
\usepackage[hmargin=2.6cm,bmargin=2.55cm,tmargin=3.05cm]{geometry}
\usepackage[usenames]{color}
\usepackage{graphbox}
\usepackage{subcaption}

\usepackage{mathtools}
\numberwithin{equation}{section}
\usepackage{amsmath,amssymb,amsthm,amsfonts,enumerate,url,mathbbol,mathrsfs,esint,tikz,graphicx,pifont,float}
\usetikzlibrary{calc}

\newcommand\pgfmathsinandcos[3]{%
	\pgfmathsetmacro#1{sin(#3)}%
	\pgfmathsetmacro#2{cos(#3)}%
}
\newcommand\LongitudePlane[3][current plane]{%
	\pgfmathsinandcos\sinEl\cosEl{#2} 
	\pgfmathsinandcos\sint\cost{#3} 
	\tikzset{#1/.style={cm={\cost,\sint*\sinEl,0,\cosEl,(0,0)}}}
}
\newcommand\LatitudePlane[3][current plane]{%
	\pgfmathsinandcos\sinEl\cosEl{#2} 
	\pgfmathsinandcos\sint\cost{#3} 
	\pgfmathsetmacro\yshift{\cosEl*\sint}
	\tikzset{#1/.style={cm={\cost,0,0,\cost*\sinEl,(0,\yshift)}}} %
}
\newcommand\DrawLongitudeCircle[3][2]{
	\LongitudePlane{\angEl}{#2}
	\tikzset{current plane/.prefix style={scale=#1}}
	\pgfmathsetmacro\angVis{atan(sin(#2)*cos(\angEl)/sin(\angEl))} %
	\draw[current plane, color=#3] (\angVis:1) arc (\angVis:\angVis+180:1);
	\draw[current plane,dashed, color=#3] (\angVis-180:1) arc (\angVis-180:\angVis:1);
}
\newcommand\DrawLatitudeCircle[3][2]{
	\LatitudePlane{\angEl}{#2}
	\tikzset{current plane/.prefix style={scale=#1}}
	\pgfmathsetmacro\sinVis{sin(#2)/cos(#2)*sin(\angEl)/cos(\angEl)}
	\pgfmathsetmacro\angVis{asin(min(1,max(\sinVis,-1)))}
	\draw[current plane, color=#3] (\angVis:1) arc (\angVis:-\angVis-180:1);
	\draw[current plane,dashed, color=#3] (180-\angVis:1) arc (180-\angVis:\angVis:1);
}

\usepackage{hyperref}
\usepackage{multirow}

\usepackage{float}

\usepackage[normalem]{ulem}
\usepackage[utf8]{inputenc}

\theoremstyle{plain}
\newtheorem{theorem}{Theorem}[section]
\newtheorem{lemma}[theorem]{Lemma}
\newtheorem{proposition}[theorem]{Proposition}
\newtheorem{corollary}[theorem]{Corollary}

\theoremstyle{definition}
\newtheorem{definition}[theorem]{Definition}
\newtheorem{example}[theorem]{Example}

\newtheorem{assumption}[theorem]{Assumption}
\newtheorem{remark}[theorem]{Remark}

\renewcommand{\epsilon}{\varepsilon}

\newcommand{\RR}{\mathbb{R}}
\newcommand{\CC}{\mathbb{C}}
\newcommand{\NN}{\mathbb{N}}    
\newcommand{\ZZ}{\mathbb{Z}}

\newcommand{\interior}{{\operatorname{int}}}

\numberwithin{equation}{section}
\numberwithin{figure}{section}

\usepackage{graphicx}

\title[Graph structure of the nodal set of eigenfunctions]{Graph structure of the nodal set and bounds on the number of critical points of eigenfunctions on Riemannian manifolds}
\author[M.~Hofmann]{Matthias Hofmann} 
\author[M.~Täufer]{Matthias Täufer}
\address{Matthias Hofmann, Fakult\"at Mathematik und Informatik, Fern\-Universit\"at in Hagen, D-58084 Hagen, Germany}
\address{Matthias Täufer, Fakult\"at Mathematik und Informatik, Fern\-Universit\"at in Hagen, D-58084 Hagen, Germany}
\thanks{\emph{Acknowledgements.} The authors would like to express their gratitude to the COST Action 18232 MAT-DYN-NET (Mathematical models for interacting dynamics on networks) for enabling numerous fruitful collaborations during the last years. The MAT-DYN-NET final conference in Braga, Portugal, in February 2024 helped to initiate this work. M. Hofmann was partially supported by the Portuguese government through FCT - Fundação para a Ciência e a Tecnologia, I.P., under the project SpectralOPs with reference 2023.13921.PEX. M. Hofmann thanks Winfried Hochstädtler for helpful discussions regarding the validity of Euler's formula for embedded graphs and Jens Wirth for suggesting Example~\ref{ex:infinite}.
The authors also thank the anonymous referee for numerous helpful remarks on an earlier version of the manuscript and in particular for drawing their attention to the Euler inequality in~\cite{KarpukhinKP-14}.}

\makeatletter
\g@addto@macro\@floatboxreset\centering
\makeatother

\begin{document}

\begin{abstract}
In this article, we  illustrate and draw connections between the geometry of zero sets of eigenfunctions, graph theory and the vanishing order of eigenfunctions. 
We identify the nodal set of an eigenfunction of the Laplacian (with smooth potential) on a compact, two-dmensional Riemannian manifolds, that is on Riemannian surfaces, as an \emph{embedded metric graph} and then use tools from elementary graph theory in order to estimate the number of critical points in the nodal set of the $k$-th eigenfunction and the sum of vanishing orders at critical points in terms of $k$ and the Euler-Poincar\'e characteristic of the surface. 
\end{abstract}

\maketitle

\section{Introduction}
    \label{sec:introduction}

In this article we illustrate and use connections between the geometry of zero sets of eigenfunctions, graph theory, vanishing order of eigenfunctions. 
Given an eigenfunction $u$ of the Laplace-Beltrami operator on a compact and connected Riemannian surface with smooth potential, its so-called \emph{nodal set} has a particular geometry: It consists of a finite number of arcs, meeting at certain points: the so-called \emph{critical points in the nodal set}.
In other words, the nodal set essentially constitutes a graph or a network, embedded in the surface, with edges given by arcs of the nodal set and vertices by critical points.

Critical points in the nodal set of eigenfunctions have been the subject of numerous investigations~\cite{Cheng-76, DonnellyF-88, DonnellyF-90} and they are closely connected to the concepts of vanishing order, doubling estimates and unique continuation in the sense that they are points where eigenfunctions can exhibit their most extreme behaviour, see e.g.~\cite{Zeldich-09, TaeuferTV-16}, and the references therein.
Most investigation so far have either focused on the Hausdorff measure of the nodal set itself or on the \emph{maximal order of vanishing} of eigenfunctions at critical points. 
Usually, these estimates are based on local techniques, zeroing in onto critical points.

In this article, we take a more global perspective and consider the nodal set as a \emph{network}, that is an embedded graph in a surface.
Together with methods from graph theory this allows us to deduce
\begin{itemize}
    \item 
    a bound on the overall number of critical points,
    \item 
    a bound on the sum over vanishing orders at critical points,
\end{itemize}
in terms of the number of nodal domains and the Euler-Poincar\'e characteristic of the surface. 
We consider the situation where the  surface has a non-empty boundary but believe that our main results are already interesting in the case without boundary, that is for closed surfaces.
To illustrate this, we demonstrate the sharpness of our estimates via examples on the sphere and the torus.

This article is structured as follows:
Section~\ref{sec:preliminaries} contains definitions, notation, and recalls the necessary results to make the intuition rigorous that nodal sets have a graph structure.
Sections~\ref{sec:results} contains our main results and their proofs: various bounds on the number of critical points and on the sum of vanishing orders. In order to obtain these bounds we use nodal partitions and study their topological features by using graph related results, most notably Euler's formula.
Finally, Section~\ref{sec:examples} demonstrates the optimality of our main results by means of several examples.

\section{Preliminaries}\label{sec:preliminaries}
\subsection{Setting the stage}

In this article, $(M,g)$ is a compact, connected $2$-dimensional Riemannian manifold, that is a \emph{Riemannian surface}.
If the boundary $\partial M$ is non-empty it is smooth and homeomorphic to a finite, disjoint union of sets homeomorphic to $S^1$, called \emph{contours}. 
We will refer to $\interior(M)= M \setminus \partial M$ as the interior of the manifold.
If $\partial M$ is empty, $M$ is simply called a \emph{closed surface}.
The family of inner products $g = (g_z)_{z \in M}$, smoothly varying in $z$, is called a \emph{Riemannian metric}
and endows $M$ with a natural notion of angles and distance.

\subsection{Function spaces and operators}

Functions $f \colon M \to \CC$ are smooth if their composition with coordinate charts are smooth functions from open sets in $\RR^2$ to $\CC$. 
Denote by $C_0^\infty(M)$ the space of smooth functions on $M$ with compact support in $\interior(M)$.
In the case of empty boundary, this is simply the space of smooth function on $M$.
We denote by $\nabla f$ their \emph{gradient}, that is the vector field satisfying 
\[
    g_z( \nabla f, X) = \partial_X f (z)
    \quad
    \text{for all vector fields $X$},
\]
where $\partial_X f$ is the directional derivative in direction $X(z)$ of $f$ in $z \in M$.
The $H^1$-Sobolev norm on $M$ is  
\[
    \lVert f \rVert_{H^1(M)}^2
    =
    \lVert f \rVert_{L^2(M)}^2
    +
    \int_{M} g(\nabla f, \nabla f),
\]
and the Sobolev space $H^1_0(M)$ is the completion of $C_0^\infty$ under this norm.
In the case of empty boundary, it is customary to simply call this space $H^1(M)$.
The (negative of the) Laplace-Beltrami operator, $- \Delta$, is the maximal self-adjoint operator, associated with the quadratic form
\[
    a(f,g)
    :=
    \int_{M} g(\nabla f, \nabla g)
    \quad
    \text{on the form domain $H^1_0(M)$}.
\]
In particular, by compactness of the embedding $H^1_0(M) \subset \subset L^2(M)$, the operator $- \Delta$ is nonnegative with purely discrete spectrum. 
In the case of non-empty boundary, $\Delta$ will impose a Dirichlet boundary condition on $\partial M$.
In this note, we stick to notation usual in the PDE community where $- \Delta$ is a nonnegative operator whereas in differential geometry, the minus is routinely omitted.

We are interested in the eigenvalue problem in $M$
\begin{equation}\label{eq:intro}
-\Delta u+ mu = \lambda u,
\end{equation}
for $m\in C^\infty(M)$
The operator $- \Delta + m$ is again self-adjoint and lower semibounded with with discrete spectrum, i.e. there is a nondecreasing sequence of eigenvalues 
\[
    - \infty < \lambda_1 < \lambda_2 \leq  \lambda_3 \leq \dots \nearrow \infty,
\]
counting multiplicities.
Furthermore, there is an associated sequence of eigenfunctions $u_1, u_2, \ldots$ forming an orthonormal basis of $L^2(M)$, cf.~\cite[Section~3.2]{Jost2017}.
By classic elliptic regularity theory, the $u_k$ are actually smooth functions on $M$.

While in some cases such as the flat torus or the sphere, discussed in Section~\ref{sec:examples}, explicit solutions of \eqref{eq:intro} can be calculated, in general less is known about the form or shape of eigenfunctions. 
One tool to better understand their geometry are their \emph{nodal sets} and \emph{critical points}. 
We define:

\begin{definition}[Nodal set and nodal domains]
For $u\in C(M)$, the \emph{nodal set} of $u$ is
\[
N(u)= \{x\in M | u(x)=0\}.
\]
The number of connected components of $N(u)$ is denoted by $n(u)$.
Furthermore, we call the connected components of $M \setminus N(u)$ the \emph{nodal domains} of $u$. 
The number of nodal domains is the \emph{nodal count}, denoted by $\mu(u)$.  
\end{definition}

We emphasize that in this notation, if $u$ is an eigenfuntion of $- \Delta + m$ and $\partial M \neq \emptyset$, then $\partial M \subset N(u)$, that is we consider the boundary as part of the nodal set.
An important result by Courant is the so-called \emph{Nodal Domain Theorem}, cf.~\cite{CourantH-66} and~\cite[VI §6]{CourantH-2008}.
\begin{proposition}[Courant's nodal domain theorem]
	\label{prop:courant}
Let $u_k$ be an eigenfunction of $- \Delta + m$ associated to $\lambda_k$.
Then $\mu(u_k) \leq k$.
\end{proposition}

\subsection{Critical points, Hartman-Wintner theorem and vanishing order}

The behaviour of solutions of functions $u$ solving $- \Delta u + mu = 0$ for $m \in C^\infty(M)$ near its zeros is well-understood:
An important concept are so-called \emph{critical points}:
\begin{definition}
\label{def:critical_point}
Let $u \in C^1(M)$.
The set of \emph{critical points in the nodal set} is 
\begin{equation}
    C(u) := \{x\in N(u) | \nabla u(x) =0\}.
\end{equation}
We also define the \emph{inner critical points in the nodal set} and the \emph{critical points on the boundary} as
\begin{equation}
    C_{\interior}(u) := C(u) \cap \interior M.
    \quad
    \text{and}
    \quad
    C_{\partial}(u) := C(u) \cap \partial M.
\end{equation}
\end{definition}
Note that for points on the boundary $\partial M$, the gradient of differentiable functions is meaningful there due to smoothness of the boundary.

\begin{remark}
    In the literature, e.g. in~\cite{Zeldich-09}, it is also common to use the name \emph{critical set}  for points where $\nabla u$ vanishes, but not necessarily $u$ itself, whereas points lying both in the nodal set and having vanishing gradient are referred to as \emph{singular points}.
    Since we are only interested in points of the nodal set, we have no need for this distinction and exclusively use the name \emph{critical points} as in Definition~\ref{def:critical_point}. 
\end{remark}

\begin{proposition}[Hartman-Wintner theorem on manifolds, cf. {\cite[Theorem 2.5]{Cheng-76}} for the case without boundary and $m = 0$] 
    \label{prop:Cheng}
    Let $M$ be a compact and connected Riemannian surface, let $m \in C^\infty(M)$, and let $u$ solve $- \Delta u + m u = 0$ with Dirichlet boundary conditions on $\partial M$.
    Then:
    \begin{enumerate}[(i)]
        \item 
        The sets $C_{\interior}(u)$ and $C_\partial(u)$ are finite.
        \item 
        The nodal set is a union of finitely many $C^2$-curves.
        These curves are either homeomorphic to the unit circle or they are homeomorphic to the unit interval and start and end at points of $C_\partial(u)$.
        \item 
        Whenever $m \geq 2$ of these curves meet at $z \in C_\interior(u)$, they locally form an equiangular system of $2m$ arcs, meeting at $z$ such that the outgoing tangent vectors have angles which are every multiple of $\pi /m$.
        \item 
        Whenever $m \geq 2$ of these curves meet at $z \in C_\partial(u)$, they locally form an equiangular system of $m+1$ arcs, meeting at $z$.
        Two arcs are contained in boundary $\partial M$ and the outgoing tangent vectors in direction of the other arcs are pointing into the surface.
        All outgoing tangent vectors along arcs at $z$ have angles which are multiples of $\pi/m$.
    \end{enumerate}
\end{proposition}

Let us comment on the steps in the proof of Proposition \ref{prop:Cheng}.
Statement (i) follows from (ii) to (iv) in combination with compactness of $M$.
Statements (ii) and (iii) follow from the Hartman-Wintner theorem~\cite{HartmanW-53, HartmanW-54}, see also \cite{HelfferHT-09}, which characterizes the behaviour of eigenfunctions (and thus also of arcs of the nodal set) of elliptic equations with smooth coefficients in open subsets of $\RR^2$ around their zeros in $\interior (M)$: 
Locally, such solutions behave like a harmonic polynomial, with a zero set consisting of a finite number of arcs, forming an equiangular system.
Since coordinate charts will locally map the equation $-\Delta u + V u = 0$ to such an elliptic equation, this settles items (ii) and (iii) of Proposition~\ref{prop:Cheng}, cf. also \cite{Bers-55} or \cite[Theorem 2.1]{Cheng-76} for a formulation appropriate in our context.

As for (iv) about points on the boundary $\partial M$, one proceeds as in~\cite[Theorem 2.3]{HelfferHT-09}:
Coordinate charts will map a neighbourhood of a boundary point to a half disk on which the equation becomes an elliptic equation of the form
    \[
    \nabla g \nabla u + V u = 0  
    \]
    with $g,V$ and $u$ smooth up to the boundary and $u$ satisfying Dirichlet boundary conditions there.
    Using a symmetric reflection of $g$, and antisymmetric reflections of $V$ and $u$, this becomes a solution of an elliptic equation on a disk, however with some loss of regularity:
    The function $u$ is now merely locally in the $H^1$ Sobolev space, the matrix-valued function $g$ is Lipschitz continuous, whereas $V$ will have a jump at the boundary of the half-disk and we can only know that it is in $L^\infty$.
    However, for this extended function $u$, the statement of the Hartman-Wintner theorem still holds, cf.~\cite[Theorem 2.3]{HelfferHT-09}, where the statement is proved for the operator $- \Delta  + V$ on domains $\Omega \subset (\RR^2)$ with $V \in L^\infty(\Omega)$.

    More precisely, the core argument in the proof of the Hartman-Wintner theorem is that solution of certain elliptic equations cannot vanish to arbitrarily high order, a statement which is nowadays known as \emph{strong unique continuation}.
    A strong unique continuation statement, sufficient for our situation ($g$ Lipschitz continuous, $V$ bounded) can be found in~\cite[Theorem 1]{KochT-09}.

\subsection{Graph structure of nodal sets}

We now make the intuition precise that nodal sets of eigenfunctions have a graph-like structure embedded in $M$ where critial points will take the role of \emph{vertices}:
The following statement is a consequence of Proposition \ref{prop:Cheng}

\begin{corollary}
    \label{cor:graph}
    Let $M$ be a compact, connected Riemannian surface, let $m\in C^\infty(M)$, and let $u$ solve $- \Delta u + m u = 0$ with Dirichlet boundary conditions on $\partial M$.
    Then, the nodal set $N(u)$ consists of a finite number of connected components.
    Each such connected component is either homeomorphic to a circle or it is homeomorphic to a finite union of closed arcs each of which is homeomorphic to the unit interval connecting only two points of $C(u)$.
    
    Every point in $C_\interior(u)$ is end point of an even number and at least four arcs, where arcs beginning and ending at the same point will be counted twice.
    
    Every point in $C_\partial(u)$ is end point of at least three arcs, where arcs beginning and ending at the same point will be counted twice.
\end{corollary}

In other words, we have identified the nodal set as a union of a finite number of circles and \emph{compact metric graphs}, embedded on the surface $M$.

\begin{definition}
    A \emph{compact metric graph} $\mathcal{G}$ is a disjoint union of intervals -- called \emph{edges} --, topologically joined at their end points -- called \emph{vertices} -- according to the structure of a combinatorial graph $G$, see \cite{Mugnolo-19} for details.

    The \emph{degree} of a vertex in metric graph is the number of edges incident to it, where loops (edges starting and ending at the same vertex) are counted twice.
\end{definition}

We are going to be interested in \emph{embedded graphs} in $M$:
\begin{definition}
Let $M$ be a compact and connected Riemannian surface.
A connected \emph{embedded graph} $G \subset M$ is a metric graph drawn on $M$ such that edges only intersect at their end points (that is at the vertices).
The connected components of $M \setminus G$, when viewed as subsets of $M$, are called the \emph{faces} of $G$.
An embedded graph $G \subset M$ is \emph{cellularly embedded} if every face is homeomorphic to a disc. 
\end{definition}

\begin{remark}
    Let us emphasize that 
    \begin{enumerate}[(i)]
        \item 
        The nodal set does not need to be connected.
        \item 
        Some components of the nodal set can be \emph{trivial}, that is homeomorphic to circles without containing critical points. 
        In order to formally turn the nodal set into an embedded metric graph, we will therefore introduce additional \emph{dummy vertices} on these components.
        \item 
        Every vertex in the metric graph, associated to a nontrivial component of the nodal set, will have degree at least four if it corresponds to a point in $C_\interior(u)$ or degree at least three if it corresponds to a point in $C_\partial(u)$.
    \end{enumerate}
\end{remark}

\begin{remark}
    It is common to consider the nodal set in the language of two-dimensional, finite \emph{CW complexes}, cf.~for instance \cite{Whitehead-49, EllisMonaghanM-13}, the appropriate framework for Euler's formula which we are going to use below.
    However, a priori, critical points, the nodal arcs and the nodal domains do \emph{not} necessarily form two-dimensional CW complexes:
    
    Firstly, faces of embedded metric graphs might not be connected.
    Worse, even the nodal set itself might be disconnected and not even a bona fide metric graph because it might have connected components, homeomorphic to circles which contain no vertices at all.
    Of course, this latter can be remedied by the introduction of additional \emph{dummy vertices} and connectedness can be achieved by insertion of additional edges, cf.~Proposition~\ref{prop:improvedeulersinequality} below.
    We also emphasize that \emph{cellularly embedded} metric graphs in $M$ with at least one vertex will indeed be CW complexes.

\end{remark}

We summarize:

\begin{lemma}
    \label{lem:adding_dummy_vertices}
    Let $M$ be a compact and connected Riemannian surface, let $m \in C^\infty(M)$, and let $u$ solve $- \Delta u + m u = 0$ with Dirichlet boundary conditions on $\partial M$. 
    Assume that the nodal set $N(u)$ is nonempty with at least one critical point and $n(u)$ many connected components.
    Then, possibly adding at most $n(u) - 1$ many additional vertices, the nodal set forms a (possibly disconnected) metric graph, embedded in $M$.
    Vertices are given by critical points of $u$ and at most $n(u) - 1$ many dummy vertices on connected components of $N(u)$ if they are homeomorphic to circles.
    Edges are given by nodal arcs connecting vertices.
\end{lemma}

\begin{definition}
    The embedded metric graph $\mathcal{G}_u \subset M$, obtained by adding vertices on all connected components of $N(u)$ as in Lemma~\ref{lem:adding_dummy_vertices} which are homeomorphic to circles is called the \emph{nodal graph}, corresponding to $u$.
\end{definition}

\subsection{Vanishing order}

By Proposition \ref{prop:Cheng}, at every critical point $z$, a finite number of nodal arcs meet at $z$.
We now relate this number to the \emph{order of vanishing}.

\begin{definition}
	Let $M$ be a Riemannian surface, $z \in M$ and $u \in C^\infty(M)$. 
	The \emph{order of vanishing} of $u$ at $z$ is 
	\[
	\Gamma_u(z)
	:=
	\max
	\left\{
		n \in \NN
		\colon
		\partial^\alpha u(z) = 0\
		\text{for all $\alpha \in \NN^d$ with $\lvert \alpha \rvert \leq N$}.
	\right\}
	\]
\end{definition}

\begin{lemma}
    \label{lem:vanishing_order}
	Let $M$ be a compact and connected Riemannian surface, let $m \in C^\infty(M)$ and let $u$ solve $- \Delta u + m u = 0$ with Dirichlet boundary conditions on $\partial M$.
    \begin{enumerate}[(i)]
        \item
        If $z \in N(u) \backslash C(u)$, then the order of vanishing of $u$ at $z$ is one.
        \item
        If $z \in C_\interior(u)$, then the order of vanishing of $u$ at $z$ is equal to $n/2$, where $n$ is the number of arcs of the nodal set, meeting at $u$.
        \item 
         If $z \in C_\partial(u)$, then the order of vanishing of $u$ at $z$ is equal to $n-1$, where $n$ is the number of arcs of the nodal set, meeting at $u$ (where we recall that two arcs, meeting in $z$ are contained in the boundary). 
        
    \end{enumerate}
        In particular at points of the critical set, the order of vanishing is at least $2$.
\end{lemma}    

Lemma \ref{lem:vanishing_order} is again a consequence of the planar Hartman-Wintner theorem in combination with the Bers scaling theorem, similar to the arguments used by Cheng in the proof of Proposition \ref{prop:Cheng}.
The notion of vanishing order is intimately connected with a similar definition in terms of the asymptotic behaviour of averages of balls, that is
\[
	\frac{\int_{B_r(z)} \lvert u(x) \rvert^p \mathrm{d} x}{r^d \int_M \lvert u(x) \rvert^p \mathrm{d} x}
\]
as $r$ tends to $0$, see~\cite{GermainMZ-24} and references therein.
The concept of vanishing order is of importance in several contexts, including doubling inequalities, unique continuation, and observability, see~\cite{DonnellyF-90, Zeldich-09, GermainMZ-24}.
In this sense, points with high vanishing order -- namely critical points -- are exactly those where eigenfunctions exhibit their most extremal behaviour.

\subsection{Euler inequality for embedded graphs}  

The well-known Euler formula for planar connected graphs
\begin{equation*}
 V  - E + F = 2,
\end{equation*}
where $V$ is the number of vertices, $E$ is the number of edges and $F$ is the number of faces can be extended to any embedded graphs via the notion of \emph{Euler-Poincar\'e characteristic} $\chi(M)$, an invariant of compact and connected Riemannian surfaces.  

\begin{proposition}[Euler-Poincar\'e characteristic and Euler formula for cellularly embedded graphs]
    Let $M$ be a compact and connected Riemannian surface. 
    Then, there is $\chi(M)$ such that for all cellularly embedded metric graphs $G \subset M$  with $V$ vertices, $E$ edges, partitioning $M \backslash G$ into $F$ faces, we have
    \[
    V - E + F = \chi(M).
    \]
\end{proposition}
In particular this justifies to define $\chi(M)$ as an invariant of the surface $M$.
On closed surfaces (i.e. without boundary) the Euler-Poincar\'e number together with the number of contours and the question of orientability or non-orientability uniquely determine the surface up to homeomorphism, see e.g.~\cite[Theorem~6.1]{GallierXu-13} and all compact surfaces have been identified~\cite{Brahana-21, Jordan-1866, 
Moebius-1861}.
In particular, for closed surfaces (that is without boundary), one has:
\begin{proposition}[Classification theorem for compact surfaces]
    \label{prop:classification_theorem}
    Let $M$ be a closed surface.
    \begin{enumerate}[(i)]
    \item 
    If $M$ is orientable, it is homeomorphic to either the sphere, or the connected sum of $g$ tori for some $g \in \{1,2,3 \dots \}$.
    \item
    If $M$ is non-orientable, it is homeomorphic to 
either to a projective plane, or a Klein bottle, or the connected sum of a projective plane or a Klein bottle with some tori.
    \end{enumerate}
\end{proposition}

In particular, on connected and closed Riemannian surfaces, the Euler-Poincar\'e characteristic is equivalent to another invariant $g (M)$, called \emph{genus} via
\[
    \chi(M) = 2 - 2 g
    \quad
    \text{if $M$ is orientable},
\]
and
\[
    \chi(M) = 2 - g
    \quad
    \text{if $M$ is non-orientable}.
\]
On orientable, closed and connected Riemannian surfaces, the genus intuitively describes the "number of holes", which means that every orientable, connected and closed Riemannian surface is homeomorphic to a manifold $S_g$, where $S_0$ denotes the sphere and $S_g$, $g \geq 1$ denotes the connected sum of $g$ tori, that is "a torus with $g$ holes", see Figure~\ref{fig:surfaces}.

\begin{figure}[ht]
        \begin{tikzpicture}
            \begin{scope}[xshift = 0cm]
                \draw[thick] (0,0) circle (1cm);
                \draw (0,-1.25) node {$S_0$};
            \end{scope}
            
            \begin{scope}[xshift = 3cm]
                \draw[thick] (0,0) ellipse (1.5cm and 1cm);
                \draw (0,-1.25) node {$S_1$};
                \begin{scope}[yshift = .1cm]
                    \draw[thick, rounded corners = 10pt] (-.5,-.1) -- (0,-.3) -- (.5,-.1);
                    \draw[thick, rounded corners = 10pt] (-.32,-.15) -- (0,.1) -- (.32,-.15);
                \end{scope}
            \end{scope}

            \begin{scope}[xshift = 7cm]
                \draw[thick, rounded corners = 18pt] 
                (-2,0) -- (-1,1.15) -- (0,.5) -- (1,1.15) -- (2,0) -- (1,-1.15) -- (0,-.5) -- (-1,-1.15) -- cycle;
                \begin{scope}[xshift = -.7cm, yshift = .1cm]
                    \draw[thick, rounded corners = 10pt] (-.4,-.1) -- (0,-.35) -- (.4,-.1);
                    \draw[thick, rounded corners = 10pt] (-.24,-.18) -- (0,.1) -- (.24,-.18);
                \end{scope}
                \begin{scope}[xshift = .7cm, yshift = .1cm]
                    \draw[thick, rounded corners = 10pt] (-.4,-.1) -- (0,-.35) -- (.4,-.1);
                    \draw[thick, rounded corners = 10pt] (-.24,-.18) -- (0,.1) -- (.24,-.18);
                \end{scope}
                \draw (0,-1.25) node {$S_2$};
            \end{scope}

            \begin{scope}[xshift = 10cm]
                    \draw (0,0) node {$\dots$};
            \end{scope}

        \end{tikzpicture}
    \caption{Closed orientable surfaces. The number $g$ denotes the genus.}
    \label{fig:surfaces}
\end{figure}

A consequence is:
\begin{proposition}[Euler inequality for embedded graphs]
    \label{prop:Euler_inequality}
    Let $M$ be a compact and connected surface and $G \subset M$ an embedded metric graph with $V$ vertices, $E$ edges, partitioning $M \backslash G$ into $F$ faces.
    Then
    \[
    V - E + F \geq \chi(M)
    \]
    with equality if and only if $G$ is cellulary embedded.
\end{proposition}

\begin{proof}    
For \emph{closed surfaces} (with empty boundary), Proposition~\ref{prop:Euler_inequality} can be found in~\cite[Section~2.3]{KarpukhinKP-14}. 
If there are non-empty boundary components, we will treat them as parts of the nodal set and add arcs in order to successively turn the graph into a cellularly embedded one:
\begin{enumerate}[(1)]
    \item
    \emph{Step 1 (Making faces have connected boundary)}:
    Add finitely many arcs to turn $N$ into an embedded metric graph $N_1$ in which the boundary of every connected component of $M^2 \backslash N_1$ is connected. This essentially means \emph{getting rid of annuli} (see Figure~\ref{fig:Step_1}).
    \item 
    \emph{Step 2 (Making faces homeomorphic to discs)}:
    Add a finite number of arcs to turn $N_1$ into an embedded metric graph $N_2$ in which every connected component of $M^2 \backslash N_2$ is homeomorphic to a disc. 
    This essentially means \emph{getting rid of handles} (see Figure~\ref{fig:Step_3}).
\end{enumerate}
Keeping track of the changes to $V$, $E$ and $F$ and applying the Euler formula, which remains valid for cellularly embedded graphs on manifolds with non-empty boundary, to the result,
we infer
    \[
        V - E + F = \chi(M) + \sum_{f} (q_f-1) + \sum_{f} g_f,
    \]
where the sums are over the set of faces, $q_f$ denotes the number of contours of a face $f$, and $$0 \leq g_f:= \begin{cases}
    1- \frac{\chi(f) + q_f}{2}, &\quad M \text{ orientiable}\\
    2- \chi(f) -q_f, &\quad M \text{ nonorientable}.
\end{cases}$$
denotes the genus of $f$. 
\end{proof}
Using $\sum_f (q_f - 1) \geq n - 1$ where $n$ denotes the number of connected components of $G$, we infer a minimal improvement of Proposition~\ref{prop:Euler_inequality}:

\begin{proposition}
\label{prop:improvedeulersinequality}
    Let $M$ be a compact and connected surface and $G\subset M$ an embedded metric graph with $n$ connected components. Suppose $G$ has  $V$ vertices, $E$ edges and partitions $M\setminus G$ into $F$ faces. Then
    \begin{equation}\label{eq:improvedeuler}
        V - E + F \geq \chi(M) + n-1.
    \end{equation}
\end{proposition}

\begin{figure}
    \begin{tikzpicture}[scale = 1.15]
        \draw[thick] (0,.5) circle (1cm);
        \draw[thick] (0,-.5) circle (1cm);
        
        \draw[fill = black, thick] (.85,0) circle (2pt);
        \draw[fill = black, thick] (-.85,0) circle (2pt);
        \draw (.5, 0) node {$v_1$};
        \draw (0,1.75) node {$G_1$};
        
        \draw[thick](.85,0) -- (4,.5);
        \draw (1.5, -.2) node {$e_1$};

        \draw (5,.5) ellipse (1cm and 1.2cm);
        \draw[fill = black, thick] (4,.5) circle (2pt);
        \draw (4.3,.5) node {$v_2$};
        \draw (5,-1.5) node {$G_2$};

        \draw (2.4,2) node {$D$};

        \draw[very thick, dotted, black!80,->, rounded corners = 5pt]
            (2.4,.4) -- (3.8, .6) arc  (175:-175:1.2cm and 1.4cm) -- (2.45,.1);
        \draw[thick, fill=black!80, black!80] (2.4,.4) circle (1pt); 
        \draw[thick, fill=black!80, black!80] (2.4,.1) circle (1pt); 
    \end{tikzpicture}

    \caption{Connecting different connected components of the nodal set as in Proposition~\ref{prop:improvedeulersinequality}. The set $D$ will remain connected.}
    \label{fig:Step_1}
    
\end{figure}

\begin{figure}
    \begin{tikzpicture}[scale = 2]
    \begin{scope}[xshift = -4cm]
        \draw[very thick] (0,0) circle (1.25cm);

        \draw[thick, dashed] (.45,.25) circle (.5cm);

        \draw[dotted, very thick, rounded corners = 5pt, black!80] (75:1.25) -- (.45,1) -- (.45,.75);
        \draw[dotted, very thick, rounded corners = 5pt, black!80] (-70:1.25) -- (.45,-.5) -- (.45,-.25);

        \draw[thick, fill = white] (75:1.25) circle (1.5pt);
        \draw[thick, fill = white] (-70:1.25) circle (1.5pt);

        \draw (.45,.25) node {$U_j$};
    \end{scope}

            \begin{scope}[]
                \draw[thick, rounded corners = 36pt] 
                (-1,.5) -- (0,.5) -- (1,1.15) -- (2,0) -- (1,-1.15) -- (0,-.5) -- (-1,-.5);

                \begin{scope}[xshift = .7cm, yshift = .1cm]
                    \draw[thick, rounded corners = 20pt] (-.4,-.1) -- (0,-.35) -- (.4,-.1);
                    \draw[thick, rounded corners = 20pt] (-.24,-.18) -- (0,.1) -- (.24,-.18);
                \end{scope}

                \draw[dashed, very thick] (-1,.5) arc (90:270:.25cm and .5cm);

                \draw[dashed] (-1,.5) arc (90:-90:.25cm and .5cm);

            \draw[dotted, very thick, black!80, rounded corners = 18pt] (-1,.5) --(0,.35) -- (1.35,.45);
            \draw[dotted, thick, black!80, rounded corners = 18pt] (-1,-.5) -- (0,-.35) -- (1.35,-.45);    

            \draw[dotted, very thick, black!80] (1.75,0) arc (0:185:.4cm and .45cm);
            \draw[dotted, thick, black!80] (.95,0) arc (180:360:.4cm and .45cm);

                 \draw[thick, black!80, fill = white]
                (1.35,.45) circle (1.5pt);
                 \draw[thick, black!80, fill = white]
                (1.35,-.45) circle (1.5pt);
            \end{scope}
        \draw (-1,0) node {$U_j$};

    \end{tikzpicture}
    \caption{Removing one handle $U_j$, which is attached at the dashed line, by adding four vertices and four edges (dotted line)}
    \label{fig:Step_3}
\end{figure}

\section{Main results}	
    \label{sec:results}

We now state and prove our main results, namely upper bounds on the number of critical points, as well as upper bounds on the sum of their vanishing orders in terms of the number of nodal domains and the Euler characteristic of the surface $M$, using the graph structure of the nodal set (see Proposition~\ref{prop:Cheng}).

\begin{definition}
    Let $M$ be a compact and connected Riemannian surface and let $u$ solve $- \Delta u + m u = 0$ for $m \in C^\infty(M)$. 
    Assume that $N(u)$ contains at least one critical point.
    We then set for the nodal graph $\mathcal{G}_u$:
    \begin{enumerate}[align=parleft]
        \item [$ V = V(u)$:]
        The number of vertices of the $\mathcal{G}_u$, that is the cardinality of the critical set $C(u)$ plus the number of connected components of $N(u)$ which are homeomorphic to a circle.
        \item [$E = E(u)$:]
        The number of edges of the $\mathcal{G}_u$, that is the number of arcs connecting critical points plus the number of connected components of $N(u)$ homeomorphic to a circle.
        \item [$F = F(u)$:]
        The number of faces of $M \backslash \mathcal{G}_u$, which coincides with the nodal count $\mu(u)$.
\end{enumerate}
\end{definition}

For the sake of convenience, we also recall:
    \begin{enumerate}[align=parleft]
        \item [$n(u)$:]
        The number of connected components of $N(u)$.
        \item [$c(u)$:]
        The number of connected components of $N(u)$ which are homeomorphic to a circle.
        \item[$\lvert C_\interior (u) \rvert$:]
        The number of critical points of $u$ in $\interior(M)$.
        \item[$\lvert C_\partial (u) \rvert$:]
        The number of critical points of $u$ on $\partial M$.     
    \end{enumerate}

\begin{theorem}
    \label{thm:main_1}
    Let $M$ be a compact and connected Riemannian surface, let $m \in C^\infty(M)$, and let $0 \not \equiv u$ solve $- \Delta u + m u  = 0$ with Dirichlet boundary conditions on $\partial M$ and assume that there is at least one critical point in the nodal set.
    Then
    \begin{equation}\label{eq:main_1}
    \lvert C_{\interior}(u)\rvert + \frac{1}{2} \lvert C_{\partial}(u) \rvert   
    \leq
    \mu(u) - \chi(M)- (n(u)-1).
    \end{equation} 
    We have equality in \eqref{eq:main_1} if and only if $\Gamma_u(z)=2$ for all $z\in C(u)$ and equality holds in \eqref{eq:improvedeuler}.
\end{theorem}
\begin{proof}
Note that in every connected component of the original nodal set $N(u)$, which is not homeomorphic to a circle, vertices are exactly points of $C(u)$ and as vertices in the embedded metric graph, the interior critical points have degree at least four and the critical points at the boundary have at least degree $3$.
We therefore obtain
\begin{equation}
\label{eq:lower_bound_E}
E-c(u)
=
\frac{1}{2} \sum_{\text{$z \in C(u)$}} \operatorname{deg}(z) 
=
\frac{1}{2} \sum_{\text{$z \in C_{\interior}(u)$}} \operatorname{deg}(z) + \frac{1}{2} \sum_{\text{$z \in C_\partial(u)$}} \operatorname{deg}(z)
\geq 
2 |C_{\interior}(u)| + \frac{3}{2} |C_\partial(u)|
\end{equation}
with equality if and only if $\Gamma_u(z)=2$ for all $z \in C(u)$ by Lemma~\ref{lem:vanishing_order}.
Using $V(u) = |C(u)| +c(u)$ and $F(u) = \mu(u)$ we conclude with Proposition~\ref{prop:improvedeulersinequality} that
\begin{equation*}
    \lvert C_{\interior}(u)\rvert + \frac{1}{2} \lvert C_{\partial}(u) \rvert \leq E - |C(u)| -c(u)
    \leq
    \mu(u) - \chi(M)- (n(u)-1),
\end{equation*}
with equality if and only if we have $\Gamma_u(z)=2$ for all $z\in C(u)$ and equality in~\eqref{eq:improvedeuler}.
\end{proof}

\begin{remark}\label{rmk:main_1_better}
The proof of Theorem~\ref{thm:main_1} is based on the observation that 
$$C_{\interior} \subset C_4(u):= \{ z \in C(u) | \deg(z) \ge 4\}.$$
Indeed, using this, one can replace \eqref{eq:main_1} by the minimally stronger inequality
\begin{equation}\label{eq:main_1alt}
|C_\interior(u)| + |C_\partial(u)\cap C_4(u)| + \frac{1}{2} |C_\partial(u) \setminus C_4(u)| \le \mu(u) - \chi(M) - (n(u)-1),
\end{equation}
an estimate which will be used to bound the number of interior critical points in Theorem~\ref{thm:contours}.
\end{remark}

A similar statement holds for the sum of degrees of critical points:

\begin{theorem}
    \label{thm:main_2}
    Let $M$ be a compact and connected Riemannian surface, let $m \in C^\infty(M)$, let $0 \not \equiv u$ solve $- \Delta u + m u  = 0$ with Dirichlet boundary conditions on $\partial M$ and assume that there is at least one critical point in the nodal set.
    Then, the sum of vanishing orders of critical points satisfies
    \[
    \sum_{z \in C_{\interior}(u)}
    \Gamma_u(z) + \frac{1}{2}\sum_{z\in C_{\partial}(u)} \Gamma_u(z)
    \leq 
    \mu(u) - \chi(M) +  |C_{\interior}(u)| + \frac{1}{2}|C_{\partial}(u)|-(n(u) -1).
    \]
    Equality holds if and only if equality holds in \eqref{eq:improvedeuler}.
\end{theorem}
\begin{proof}
    By Lemma~\ref{lem:vanishing_order} we have
    \begin{equation}\label{eq:degordformula}
        \deg(z) = \begin{cases} 2\Gamma_u(z), &\quad z \in C_{\interior}(u),\\
        \Gamma_u(z) +1, &\quad z \in C_{\partial}(u). \end{cases}
    \end{equation}
    Proceeding as in~\eqref{eq:lower_bound_E}, we have
    \begin{align*}
    E-c(u)
    &=
    \frac{1}{2} \sum_{\text{$z \in C_{\interior}(u)$}} \operatorname{deg}(z) + \frac{1}{2} \sum_{\text{$z \in C_\partial(u)$}} \operatorname{deg}(z)
    = \sum_{z\in C_{\interior}(u)} \Gamma_u(z) + \frac{1}{2} \sum_{z\in C_{\partial}(u)} \Gamma_u(z) + \frac{1}{2} |C_\partial(u)|.
    \end{align*} 
Using $V(u) = |C(u)| +c(u)$ and $F(u) = \mu(u)$ we conclude with Proposition~\ref{prop:improvedeulersinequality} that
    \begin{align*}
\sum_{z\in C_{\interior}(u)} \Gamma_u(z) + \frac{1}{2} \sum_{z\in C_{\partial}(u)} \Gamma_u(z) & = E - c(u)- \frac{1}{2} |C_\partial(u)| \\
&\le  V+ F -\chi(M) -c(u)- \frac{1}{2} |C_\partial(u)|-(n(u) -1) \\
&=  \mu(u) - \chi(M) +  |C_{\interior}(u)| + \frac{1}{2}|C_{\partial}(u)| -(n(u) -1)
    \end{align*}
    with equality if and only if equality holds in \eqref{eq:improvedeuler}.
\end{proof}

We now formulate some consequences.
The next statement provides an upper bound only depending on the number of nodal domains $\mu(u)$.

\begin{corollary}\label{cor:main_2}
    Let $M$ be a compact and connected Riemannian surface, let $m \in C^\infty(M)$, let $0 \not \equiv u$ solve $- \Delta u + m u  = 0$ with Dirichlet boundary conditions on $\partial M$ and assume that there is at least one critical point in the nodal set.
    Then, 
    \begin{equation}\label{eq:main_2cor_1}
    \lvert C_{\interior}(u)\rvert + \frac{1}{2} \lvert C_{\partial}(u) \rvert 
    \leq
    \mu(u) - \chi(M),
    \end{equation}
    and
    \begin{equation}
    \label{eq:main_2cor_2}
   \sum_{z \in C_{\interior}(u)}
    \Gamma_u(z) + \frac{1}{2}\sum_{z\in C_{\partial}(u)} \Gamma_u(z)
    \leq 
    2 \mu(u) - 2\chi(M). 
    \end{equation}
    Equality holds in \eqref{eq:main_2cor_1} and \eqref{eq:main_2cor_2}, respectively, if and only if $N(u)$ is cellularly embedded and $\Gamma_u(z)=2$ for all $z\in C(u)$.
\end{corollary}

Together with Courant's Nodal Domain Theorem, Proposition \ref{prop:courant}, we obtain:
\begin{corollary}\label{cor:1}
    Let $M$ be a compact and connected Riemannian surface, let $m \in C^\infty(M)$, and let $u_k$ be an eigenfunction, associated to the $k$-th eigenvalue $\lambda_k$ of $- \Delta + m$ with at least one critical point in the nodal set.
    Then
    \begin{equation}
    \sum_{z \in C_{\interior}(u)}
    \Gamma_u(z) + \frac{1}{2}\sum_{z\in C_{\partial}(u)} \Gamma_u(z)
    \leq
    k - \chi(M),
    \end{equation}
    and
    \begin{equation}
     \sum_{z \in C_{\interior}(u)}
    \Gamma_u(z) + \frac{1}{2}\sum_{z\in C_{\partial}(u)} \Gamma_u(z)
    \leq
    2 k - 2\chi(M).
    \end{equation}
\end{corollary}

Using that on compact and connected Riemannian surfaces, the $k$-th eigenvalue is asymptotically proportional to $k$ by Weyl's law, we obtain:

\begin{corollary}
    \label{cor:3}
    Let $M$ be a compact and connected Riemannian surface and let $m \in C^\infty(M)$.
    Then there is $c >0$ such that for sufficiently large $\lambda$ and every eigenfunction $u$ of $- \Delta + m$ to the eigenvalue $\lambda$, we have
    \[
    \sum_{z \in C_{\interior}(u)}
    \Gamma_u(z) + \frac{1}{2}\sum_{z\in C_{\partial}(u)} \Gamma_u(z)
    \leq
    c \lambda.
    \]
\end{corollary}

\begin{proof}[Proof of Corollary~\ref{cor:3}]
    By Weyl's asymptotic law, the $k$-th eigenvalue is proportional to $\lambda$.
    The statement then follows from Corollary~\ref{cor:1}.
\end{proof}

\begin{remark}
    Corollary~\ref{cor:3},  should be compared to~\cite{DonnellyF-88} where it is proved that for an eigenfunction $u$ to the eigenvalue $\lambda$, one has
    \[
    \max_{z \in C(u)}
    \Gamma_u(z)
    \leq
    c \sqrt{\lambda}.
    \]
\end{remark}

Let us conclude the section with bounds on the number of interior critical points.

\begin{theorem}
    \label{thm:contours}
    Let $M$ be a compact, connected Riemannian surface and $m \in C^\infty(M)$. Denote the number of contours by $q(M)$. 
    \begin{enumerate}[(i)]
        \item 
        If $u$ solves $-\Delta u + m u = 0$ with Dirichlet boundary conditions on $\partial M$ then 
        \[
        \lvert C_{\interior}(u) \rvert
        \leq
        \mu(u)
        -
        \chi(M)- q(M),
        \quad
        \text{and}
        \quad
        \sum_{z \in C(u)}
        \Gamma_u(z)
        \leq
        2 \mu(u)
        -
        2 \chi(M)-2q(M),
        \]
        \item 
        If $u_k$ is an eigenfunction to the $k$-th eigenvalue $\lambda_k$ of $-\Delta + m$, then        
        \[
        \lvert C_\interior(u_k) \rvert
        \leq
        k
        -
        \chi(M)-q(M)
        \quad
        \text{and}
        \quad
        \sum_{z \in C(u)}
        \Gamma_u(z)
        \leq
        2 k - 2 \chi(M)-2q(M).
        \]
    \end{enumerate}
\end{theorem} 
\begin{remark}
Note, that the upper bound in Theorem~\ref{thm:contours} only depends on the genus $g$ of the Riemannian surface. In particular, we have
\begin{equation}
    \chi(M) + q(M) = \begin{cases} 2-2g, &\quad M \text{ orientable,}\\
    1-g, &\quad M \text{ not orientable.} \end{cases}
\end{equation}
\end{remark}
\begin{proof}[Proof of Theorem~\ref{thm:contours}]
The restriction of the nodal set to any contour is either homeomorphic to a circle or contains a critical point $z$. In the latter case, since the eigenfunction needs to change sign at every nodal arc, either $\deg(z) \ge 4$ or there exists a second critical point on the same contour. In other words
\[
|C_\partial(u) \cap C_4(u) | + \frac{1}{2} |C_\partial(u) \setminus C_4(u)| + (n(u)-1) \ge q(M).
\]
Similarly, we conclude
\[
\frac{1}{2} \sum_{z\in C_\partial(u)} \Gamma_u(z) + \frac{1}{2} |C_\partial(u) \cap C_4(u)| + 2(n(u) -1) \ge 2q(M). 
\]
The statement now follows from Remark~\ref{rmk:main_1_better} and Theorem~\ref{thm:main_2}.
\end{proof}

\begin{remark}
    We believe that the above results are already interesting in the case of empty boundary, that is for \emph{closed surfaces}.
    For instance, we infer from Theorem~\ref{thm:main_1} and Corollary~\ref{cor:main_2}
    \[
    \lvert C (u) \rvert
    \leq 
    \mu(u) - \chi(M)
    \quad
    \text{and}
    \quad
    \sum_{z \in C(u)}
    \Gamma_u(z)
    \leq 
    \mu(u) - \chi(M).
    \]
\end{remark}

\section{Examples}
    \label{sec:examples}

In this section we discuss examples in which an eigenbasis can be explicitly computed.

\subsection{On the sphere}
 Let us consider the sphere $M=S_0$. In this case, the eigenfunctions of the spherical Laplacian are known as the \emph{spherical harmonics} (see \cite[Section~14.30]{NIST:DLMF} and \cite[Section~VII.5]{CourantH-66}). They are used for example to describe the angular part of wavefunctions in electrons in atoms and can be used to solve the Schrödinger equation for the hydrogen atom and is used to describe the orbital model for atoms.

\begin{figure}[ht]
\centering
\begin{minipage}{.4\textwidth}
	\begin{tikzpicture} 
		\def\R{2.5} 
		\def\angEl{35} 
		\draw (0,0) circle (\R);
		\foreach \t in {-80,-60,...,80} { \DrawLatitudeCircle[\R]{\t}{blue} }
		\foreach \t in {-5,-35,...,-175} { \DrawLongitudeCircle[\R]{\t}{gray}}
		\node[draw,circle,minimum size=5cm](hdm){};
		\DrawLongitudeCircle[\R]{115}{blue}
	\end{tikzpicture}
 \end{minipage}\hspace{1em}
 \begin{minipage}{.4\textwidth}
	\begin{tikzpicture} 
		\def\R{2.5} 
		\def\angEl{35} 
		\draw (0,0) circle (\R);
		\foreach \t in {-80,-60,...,80} { \DrawLatitudeCircle[\R]{\t}{gray} }
		\foreach \t in {-5,-35,...,-175} { \DrawLongitudeCircle[\R]{\t}{blue}}
		\node[draw,circle,minimum size=5cm](hdm){};
	\end{tikzpicture}
\end{minipage}
\caption{Different nodal sets of spherical harmonics. The left is an example for which the optimal bounds for the number of critical points. The right spherical hamonics function does not satisfy the optimal bound for the number of critical points, however the vanishing orders satisfy the optimal bound.}
\label{fig:sphericalfunctions}
\end{figure}

A complete system of eigenfunctions can be determined  via
\begin{equation}\label{eq:sphericalharmonics}
    Y_{\ell,m}(\phi, \theta) = \begin{cases}
        P_\ell^0(\cos(\theta)), \quad m=0\\
        P_\ell^m(\cos(\theta)) \cos(m\phi), \quad 1\le m \le \ell\\
        P_\ell^{|m|}(\cos(\theta)) \sin(|m|\phi), \quad -\ell \le m \le -1
    \end{cases}
\end{equation}
for $\ell=0,1,2, \ldots$  with $m=-\ell, \ldots, 0, 1, \ldots, \ell$, where $P_\ell^m(\cos(\theta))$ is the associate Legendre polynomial  (see \cite[Section~VII.5]{CourantH-66}). The corresponding eigenvalues are $\lambda_{\ell}= \ell (\ell+1)$ and $Y_{\ell,m}$ for $m=-\ell, \ldots, 0, 1, \ldots, \ell$ form a basis of the corresponding eigenspace. The associated Legendre polynomial $P_\ell^{|m|}$ has $\ell-|m|$ roots in $(-1,1)$ (see \cite[Section~14.16]{NIST:DLMF} or \cite[Chapter~IX]{hobson1931theory}) and due to the form of \eqref{eq:sphericalharmonics} the node set of the spherical harmonics consists of $m$ longitudes and  $\ell-m$ latitudes intersecting in $2m(\ell-m)$ points.  The number of critical points, vanishing order and nodal count of the spherical harmonic functions $Y_{\ell,m}$ can be easily obtained (see Figure~\ref{fig:sphericalfunctions}) In summary, $Y_{\ell, m}$ has
\begin{itemize}
    \item one nodal domain for $\ell=0$, two nodal domains for $\ell=1$, $\ell+1$ nodal domains for $m=0$ and $\ell \ge 2$, and $2m(\ell-m+1)$ otherwise;
    \item $2m(\ell-m)$ critical points if $|m|\le 1$. Otherwise $2m(\ell-m)+2$ critical points;
    \item The $2m(\ell-m)$ critical points are vanishing of order $2$ if $|m|\le 1$. For $|m|\ge 2$ there exist two additional critical points of order $m$.
\end{itemize}
The nodal partitions are cellularly embedded. 
We find for $\ell \ge 2$ and $m\neq 0$
\begin{equation}
\lvert C(Y_{\ell, m}) \rvert \le 2m(\ell-m+1) -2.
\end{equation}
One easily verifies that equality holds if $m=1,2$ and the bounds in Theorem~\ref{thm:main_1} and Theorem~\ref{thm:main_2} are therefore sharp. On the other hand, for $m\ge 3$ equality cannot hold due to the existence of critical points with vanishing order $m\ge 3$.

\subsection{On the torus} 

Next, we show that the bounds in Theorems~\ref{thm:main_1} and~\ref{thm:main_2} are also optimal on the torus $S_1$.
More precisely, we consider the \emph{flat torus}, that is, we can identify $S_1$ with $\RR^2 \operatorname{mod} 2 \pi \ZZ^2$ with the Riemannian metric equal to the usual Euclidean scalar product at every point.
Alternatively, we can interpret the Laplacian on $S_2$ with this metric as the Laplacian on the square $(- \pi, \pi)^2$ with periodic boundary conditions.

The spectrum and eigenfunctions are then explicitly known and the following lemma is folklore:

\begin{lemma}
    \label{lem:eigenspaces_torus}
    An orthonormal basis of eigenfunctions of $-\Delta$ on $(- \pi, \pi)^2$ with periodic boundary conditions is given by $(e_l)_{l \in \ZZ^2}$, where 
    \[
    e_l(x)
    =
    \frac{1}{\sqrt{2 \pi}}
    \exp (x \cdot l).
    \]
    The eigenvalue, corresponding to $e_l$, is $\lvert l \rvert^2 = l_1^2 + l_2^2$.
\end{lemma}
In particular, an eigenvalue $\lambda$ can have high multiplicity if many lattice points $l \in \ZZ^2$ lie on the circle with radius $\sqrt{\lambda}$ around zero.
By taking linear combinations of the four elements of the form $e_{(\pm l_1, \pm l_2)}$, we can also see that functions of the form
    \[
    u(x_1, x_2)
    =
    \cos(l_1 x_1) \cos(l_2 x_2)
    \]
    are eigenfunctions.
    However, it is easy to see that the nodal set of these eigenfunction will have exactly $4 \cdot l_1 \cdot l_2$ many critical points and equally many nodal domains, all of vanishing order $2$ (cf. Figure~\ref{fig:nodal_domains_torus} for an illustration).
    This shows that the upper bound of Theorem~\ref{thm:main_1} is sharp on the torus.
    \begin{figure}
        \begin{tikzpicture}
            \draw[thick, dashed] (-2,-2) rectangle (2,2);

            \draw[thick, dotted]
            (-1.5,-2) -- (-1.5,2);
            \draw[thick, dotted]
            (-.5,-2) -- (-.5,2);
            \draw[thick, dotted]
            (.5,-2) -- (.5,2);
            \draw[thick, dotted]
            (1.5,-2) -- (1.5,2);

            \draw[thick, dotted]
            (-2,-1) -- (2,-1);
            \draw[thick, dotted]
            (-2,1) -- (2,1);

            \draw[thick, fill = white]
            (-1.5,-1) circle (2pt);
            \draw[thick, fill = white]
            (-.5,-1) circle (2pt);
            \draw[thick, fill = white]
            (.5,-1) circle (2pt);
            \draw[thick, fill = white]
            (1.5,-1) circle (2pt);

             \draw[thick, fill = white]
            (-1.5,1) circle (2pt);
            \draw[thick, fill = white]
            (-.5,1) circle (2pt);
            \draw[thick, fill = white]
            (.5,1) circle (2pt);
            \draw[thick, fill = white]
            (1.5,1) circle (2pt);
            
            \draw (-1,1.5) node {$D_1$};
            \draw (0,1.5) node {$D_2$};
            \draw (1,1.5) node {$D_3$}; 
            \draw (1.75,1.5) node {$D_4$};
            
            \draw (-1,0) node {$D_1$};
            \draw (0,0) node {$D_2$};
            \draw (1,0) node {$D_3$}; 
            \draw (1.75,0) node {$D_4$};             
        \end{tikzpicture}

    \caption{Nodal set of the eigenfunction $u(x_1,x_2) = \cos(2 x_1) \cos(x_2)$ on the flat torus (opposides side of the square are identified) with eight critical points and eight nodal domains, illustrating sharpness of the upper bound of Theorem~\ref{thm:main_1}.}
    \label{fig:nodal_domains_torus}
    \end{figure}

We have seen before that on the sphere $S_0$ we can also find points of arbitrarily high degree in the nodal graph of eigenfunctions -- namely at the "north" and "south poles".
Let us demonstrate that one can observe a similar phenomenon on the flat torus.

\begin{theorem}[Cf.~\cite{Taeufer-17}]
    \label{thm:high_vanishing_order}
    On the flat torus $S_1$, for every point $x \in S_1$ and every $n \in \NN$, there is an eigenfunction $u$ of the Laplacian which has a point of vanishing order at least $n$ on which at least $2n$ arcs of the nodal set meet.
\end{theorem}

We also refer to~\cite{GermainMZ-24} for more recent developments on the vanishing order of eigenfunction on the flat torus.
The key is the sums-of-squares theorem by Gauss \cite{Gauss-1801}; see also \cite[Chapter 1]{Fricker-82}:
\begin{proposition}
    \label{prop:sum_of_squares}
    Let $n \in \NN$ have the prime factor decomposition
    \[
    n = p_1^{a_1} \cdot \dots \cdot p_k^{a_k} \cdot q_1^{b_1} \cdot \dots \cdot q_l^{b_l} \cdot 2^c
    \]
    where the $p_i$ are primes of the form $4k + 1$, and the $q_j$ are primes of the form $4k + 3$.
    Then the number of pairs $(l_1, l_2)$ satisfying $l_1^2 + l_2^2 = n$ is equal to
    \[
  \begin{cases}
   4 \cdot (1 + a_1) \cdot \dots \cdot (1 + a_n) 
   \quad
    & \text{if all $b_i$ are even},\\
   0
   & \text{else}.
  \end{cases}
  \]
\end{proposition}

\begin{proof}[Proof of Theorem~\ref{thm:high_vanishing_order}]
By Proposition~\ref{prop:sum_of_squares} and the spectral decomposition in Lemma~\ref{lem:eigenspaces_torus} we infer that the Laplacian on the flat torus has eigenspaces of arbitrarily high dimension.
Due to the spectral decomposition of Lemma~\ref{lem:eigenspaces_torus}, every eigenfunction $u$ to an eigenvalue $\lambda \geq 0$ is of the form
\[
    u = \sum_{k \in I_\lambda} \mu_k e_k, \quad \mu_k \in \CC,
\]
where $I_\lambda = \{ k \in \ZZ^2 \colon \lvert k \rvert^2 = \lambda \}$.
Expanding the $e_k$ in a Taylor series around $0$, we find
 \[
 e_k(x) 
 =
 \sum_{\alpha \in \NN_0^d} \frac{1}{\alpha!} (D^\alpha e_k)(0) \cdot x^\alpha
 =
 \sum_{\alpha \in \NN_0^d} \frac{(i k)^\alpha}{\alpha!} x^\alpha
\]
where we used multindex notation, i.e. for $\alpha = (\alpha_1, ..., \alpha_d) \in \NN_0^d$, we write
 $\alpha!  := \alpha_1 ! \cdot \dots \cdot \alpha_d!$, 
 $D^\alpha := \partial_{x_1}^{\alpha_1} \dots \partial_{x_d}^{\alpha_d}$,
 $k^\alpha := k_1^{\alpha_1} \cdot \dots \cdot k_d^{\alpha_d}$,
 $\lvert \alpha \rvert_1 := \alpha_1 + \dots + \alpha_d$.
Then, $u$ can be written as
\[
 u(x)
 =
 \sum_{k \in I_\lambda} \mu_k e_k
 =
 \sum_{\alpha \in \NN_0^d} \frac{(i x)^\alpha}{\alpha!} 
  \left(
   \sum_{k \in I_\lambda} \mu_k k^\alpha
  \right).
\]
Since the Taylor series is locally absolutely convergent, $u$ vanishes to order $n$ at $0$ if for all $\alpha \in \NN_0^d$ with $\lvert \alpha \rvert_1 \leq n$, we have
\[
 \sum_{k \in I_\lambda} \mu_k k^\alpha = 0.
\]
This is a system of finitely many linear equations, indexed by $\alpha$, with variables $\{ \mu_k \}_{k \in I_\lambda}$.
More precisely
\begin{align*}
 \sharp \{ \text{equations} \} 
 &= 
 \sharp \{ \alpha \in \NN_0^d \colon \lvert \alpha \rvert_1 \leq N \}  =: C(n),
 \quad \text{\emph{fixed}, once we chose $n$},\\
 \sharp \{ \text{variables} \} 
 &=
 \sharp
 \{ k \in \ZZ^d \colon \lvert k \rvert^2 = \lambda \}
 =
 \# I_\lambda.
\end{align*}
We clearly can find a nontrivial solution $\{ \mu_k \}_{k \in I_\lambda}$, if for a given $C = C(n) \in \NN$ we can find $\lambda \geq 0$ such that $\sharp I_\lambda \geq C$. 
This will then yield a function $u$ which vanishes to order $n$ at $0$.
But by Proposition~\ref{prop:sum_of_squares}, we have that $I_{5^C}$ contains $4C$ many points, i.e. we can choose $\lambda = 5^C$.
The fact that at least $2n$ arcs meet at the critical point follows from Lemma~\ref{lem:vanishing_order}.
\end{proof}

From the proof of Theorem~\ref{thm:high_vanishing_order} it is obvious that the construction of functions with high vanishing order readily generalizes to higher dimensional flat tori.
However, due to the more complicated structure of the nodal set in higher dimensions, there is no longer a simple interpretation of the degree of vanishing as (half the) degree of a vertex in a graph.

\subsection{Nodal domains on bounded domains in the plane}

In this subsection, we discuss the special case of planar, bounded domains.
Throughout this subsection, we stipulate:
\begin{assumption}\label{ass:set-of-assumptions}
 $\Omega \subset \RR^2$ is compact, satisfies the interior cone property, and has piecewise $C^{1,\alpha}$ boundary for some $\alpha > 0$; this means that the boundary curve is locally the graph of a differentiable function with $\alpha$-H\"older continuous derivative.
\end{assumption}

For $m \in L^\infty(\Omega)$, consider the eigenvalue problem
\begin{equation}\label{eq:intro2}
\begin{cases}
-\Delta u+ mu = \lambda u, &\qquad \text{ in }\Omega\\
u= 0, &\qquad \text{ in } \partial \Omega,
\end{cases}
\end{equation}

 The analogue of Proposition~\ref{prop:Cheng} now holds for solutions to eigenvalue problems \eqref{eq:intro2} (see e.g. \cite[Section 2]{HelfferHT-09}). 

\begin{proposition}[Hartman-Wintner theorem on domains] 
    \label{prop:Cheng-domains}
    Suppose Assumption~\ref{ass:set-of-assumptions} is satisfied and let $u$ be a solution to \eqref{eq:intro2}, then 
    \begin{enumerate}[(i)]
        \item 
        The sets $C_{\interior}(u)$ and $C_\partial(u)$ are finite.
        \item 
        The nodal set is a union of finitely many $C^{1,\varepsilon}$-curves for some $\varepsilon>0$.
        These curves are either homeomorphic to the unit circle or they are homeomorphic to the unit interval and start and end at points of $C_\partial(u)$.
        \item 
        Whenever $m \geq 2$ of these curves meet at $z \in C_\interior(u)$, they locally form an equiangular system of $2m$ arcs, meeting at $z$ such that the outgoing tangent vectors have angles which are every multiple of $\pi /m$.
        \item 
        Whenever $m \geq 2$ of these curves meet at $z \in C_\partial(u)$, they locally form an equiangular system of $m+1$ arcs, meeting at $z$.
        Two arcs are contained in boundary $\partial M$ and the outgoing tangent vectors in direction of the other arcs are pointing into the surface.
        All outgoing tangent vectors along arcs at $z$ have angles which are multiples of $\pi/m$.
    \end{enumerate}
\end{proposition}
We can recover our main results in this case under milder assumptions on the regularity of $\Omega$ and $m$. 
Denoting by $q$ the number of contours of $\Omega$, we have
\[
\chi(M) = 2- q.
\]
In particular, the analogue of Theorems~\ref{thm:main_1} and Theorem~\ref{thm:main_2} are:
\begin{theorem}\label{thm:main_domain}
  Let $u$ be a nontrivial solution of \eqref{eq:intro2} under Assumption~\ref{ass:set-of-assumptions}. Then, 
    \begin{align}\label{eq:main_2cor_1domain}
    \lvert C_{\interior}(u)\rvert + \frac{1}{2} \lvert C_{\partial}(u) \rvert 
    &\leq
    \mu(u) +q-2 - (n(u) -1)
    \end{align}
    and
    \begin{align}
    \label{eq:main_2cor_2domain}
   \sum_{z \in C_{\interior}(u)}
    \Gamma_u(z) + \frac{1}{2}\sum_{z\in C_{\partial}(u)} \Gamma_u(z) 
    &\leq 
    2 \mu(u)+ 2q -4- 2(n(u) -1). 
    \end{align}
    Equality holds in \eqref{eq:main_2cor_1domain} and \eqref{eq:main_2cor_2domain}, respectively, if and only if we have equality in~\eqref{eq:improvedeuler} and $\Gamma_u(z)=2$ for all $z\in C(u)$.
\end{theorem}

We can obtain then bounds for the interior critical points via Theorem~\ref{thm:contours}.
\begin{corollary}
      Let $u$ be a nontrivial solution to \eqref{eq:intro2} under Assumption~\ref{ass:set-of-assumptions}. Then,
    \begin{equation*}
    \lvert C_{\interior}(u) \rvert 
    \leq
    \mu(u) -2 
    \quad
    \text{and}
    \quad
   \sum_{z \in C_{\interior}(u)}
    \Gamma_u(z) 
    \leq 
    2 \mu(u)-4.
    \end{equation*}
\end{corollary}

Our final example illustrates the importance of the finiteness of the number of contours in order to even ensure finiteness of the number of critical points.
\begin{example}\label{ex:infinite}
Consider the first Dirichlet eigenfunction on domain $\Omega_+$ which consists of a half disk with an infinite sequence of half-disks of decreasing radii removed, see Figure~\ref{infinite_mirroring}. 
By mirroring $\Omega_+$ and antisymmetric reflection we obtain an eigenfunction on the disk with infinitely many holes removed. It has infinitely many critical points on the axis of symmetry.
\begin{figure}[ht]
\centering
\begin{minipage}{0.4\textwidth}
\includegraphics[scale=0.5]{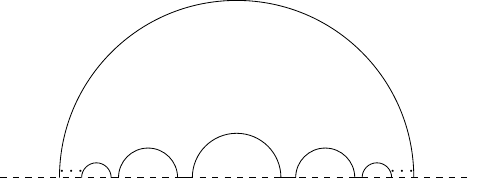}
\end{minipage} 
\begin{minipage}{0.4\textwidth}
\includegraphics[scale=0.5]{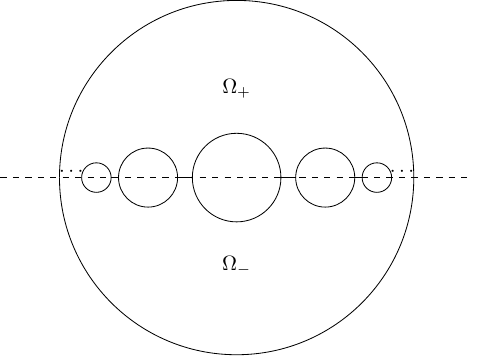}
\end{minipage}
\caption{\small On the left a semi-disk with infinitely many semi-disks cut out and the mirrored image on the right. 
The corresponding mirrored first Dirichlet eigenfunction is an eigenfunction on the mirrored domain with nodal domains $\Omega_+$ and $\Omega_-$.}\label{infinite_mirroring}
\end{figure}
However, the resulting domain has infinitely many contours, so our previous results do not apply.
Yet, it would be interesting to investigate whether in this case, one can still provide bounds on the number of \emph{interior} critical points.
This is a subject for future investigations.
\end{example}

\end{document}